\documentclass[preprint%
]{elsarticle}
\usepackage{algorithm, algorithmic}
\usepackage[latin1]{inputenc}
\usepackage{amsmath}\usepackage{amsthm}\usepackage{amssymb}\usepackage{mathrsfs}\usepackage{amscd}\usepackage{graphicx}\usepackage{subfigure}\usepackage{amsfonts}\usepackage{amsxtra}\usepackage{color}
\usepackage{multirow}
\usepackage{amscd}
\usepackage{setspace} 

\usepackage[T1]{fontenc}

\DeclareMathOperator{\FFP}{FFP}

\DeclareMathOperator{\Gram}{Gram}

\DeclareMathOperator{\trace}{trace}

\DeclareMathOperator{\Id}{Id}

\DeclareMathOperator{\GG}{\mathcal{G}}
\DeclareMathOperator{\FFF}{\mathcal{F}}
\DeclareMathOperator{\Pol}{Pol}
\DeclareMathOperator{\1}{\mathbf{1}}

\theoremstyle{definition}\newtheorem{definition}{Definition}[section]\newtheorem{example}[definition]{Example}\newtheorem{remark}[definition]{Remark}
\theoremstyle{plain}\newtheorem{theorem}[definition]{Theorem}\newtheorem{lemma}[definition]{Lemma}\newtheorem{corollary}[definition]{Corollary}\newtheorem{proposition}[definition]{Proposition}

\newcommand{\R}{\mathbb{R}}\newcommand{\C}{\mathbb{C}}

\newcommand{\Gkd}{{\mathcal G}_{k,d}}
\newcommand{\Gld}{{\mathcal G}_{\ell,d}}
\newcommand{ \God}{{\mathcal G}_{1,d}}

\journal{Applied and Computational Harmonic Analysis
}

\begin{document}

\begin{frontmatter}

\title{Tight $p$-fusion frames}

\author[a]{C.~Bachoc}
\ead{Christine.Bachoc@math.u-bordeaux1.fr}
\address[a]{Univ. Bordeaux, IMB, UMR 5251, F-33400 Talence, France}

\author[b]{M.~Ehler}
\ead{martin.ehler@helmholtz-muenchen.de}

\address[b]{Helmholtz Zentrum M\"unchen, 
Institute of Biomathematics and Biometry,  
Ingolst\"adter Landstrasse 1,  
85764 Neuherberg, Germany}

\begin{abstract}
Fusion frames enable signal decompositions into weighted linear subspace components.  
For positive integers $p$, we introduce $p$-fusion frames, a sharpening of the notion of fusion frames. 
Tight $p$-fusion frames are closely related to the classical notions of designs and cubature formulas in Grassmann spaces and are analyzed with methods from harmonic analysis in the Grassmannians. 
We define the $p$-fusion frame potential, derive bounds for its value, and discuss the connections to tight $p$-fusion frames.


\end{abstract}

\begin{keyword} fusion frame potential, Grassmann space, cubature formula, design, equiangular, simplex bound.
\end{keyword}


\end{frontmatter}

\section{Introduction}

In modern signal processing, basis-like systems are applied to derive stable and redundant signal representations. Frames are basis-like systems that span a vector space but allow for linear dependency, that can be used to reduce noise, find sparse representations, or obtain other desirable features unavailable with orthonormal bases \cite{Christensen:2003aa,Ehler:2010aa,Ehler:2010ac,Okoudjou:2010aa,Okoudjou:2010xx}. In fusion frame theory as introduced in \cite{Casazza:2008aa}, see also \cite{Bodmann:2007fk,Boufounos:2011uq,Calderbank:2010aa,Casazza:2011aa,Gavrutka}, the signal is projected onto a collection of linear subspaces that can represent, for instance, sensors in a network \cite{Rozell:2006uq} or nodes in a computing cluster \cite{Bjorstad:1991kx}. To obtain a signal reconstruction that is robust against noise and data loss, the subspaces are usually chosen in some redundant fashion and, as such, fusion frames are tightly connected to coding theory \cite{Bodmann:2007fk,Bodmann:2007kx}. 

Tight fusion frames \cite{Casazza:2008aa}  provide a direct reconstruction formula, and can be characterized as the minimizers of the fusion frame 
potential. The error caused from the loss of one or two subspaces  within a tight fusion frame is minimized for equidimensional subspaces that satisfy the simplex bound with equality \cite{Kutyniok:2009aa}. The simplex bound, as derived in 
\cite{Conway:1996aa}, is an extremal estimate on the maximum of the inner products $\langle P_V, P_W\rangle:=\trace(P_VP_W)$  between the projectors associated to equidimensional linear subspaces $V$ and $W$. Equality in this bound implies that the subspaces are equiangular, meaning that the inner products between distinct pairs take the same value. 

We derive a generalized simplex bound that also holds for subspaces whose dimensions can vary. Equality holds if and only if the fusion frame is tight and equiangular. 
In Section \ref{subsec:maximal}, we prove  that the number of equiangular subspaces in $\R^d$ cannot exceed $\binom{d+1}{2}$, generalizing Gerzon's bound for the maximal number of equiangular lines \cite{LemmensSeidel:1973aa}. 

The $p$-fusion frame potential of a collection of subspaces is introduced as an extension of the fusion frame potential discussed in \cite{Casazza:2009aa,Massey:2010fk},  as well as a convenient $\ell^p$ approximation 
of  the maximum among the inner products $\langle P_V, P_W\rangle$ of
pairwise distinct subspaces. We moreover derive a bound for the
$p$-fusion frame potential that yields the simplex bound at the
limit. 

We  introduce  the notions of $p$-fusion frames and tight $p$-fusion
frames, where $p\geq 1$ is an integer. These notions generalize the
notion of (tight)  fusion frames corresponding  to the case  $p=1$. For subspaces of equal
dimension, we apply methods  from harmonic
analysis on Grassmann spaces in order to analyse these objects. 
In particular we characterize tight $p$-fusion frames by the
evaluation of certain multivariate Jacobi polynomials at  the
principal angles between subspaces.
Moreover we relate  them to cubature formulas in Grassmann space.

A general framework for cubatures in polynomial spaces is proposed in \cite{Harpe:2005aa}. Designs for the Grassmann space, i.e., cubatures with constant weights, have been introduced and studied in  \cite{Bachoc:2005aa,Bachoc:2002aa}. 
We prove that cubatures of strength $2p$  can  be characterized as the
minimizers of the $p$-fusion frame potential. The notions of tight $p$-fusion frames and of cubatures
of strength $2p$ coincide for $p=1$; however the latter is stronger than the former for $p\geq 2$.

We verify the existence of tight $p$-fusion frames for any integer $p\geq 1$ by using results in \cite{Harpe:2005aa}.  Moreover, we present general constructions of tight $p$-fusion frames. One is based on orbits of finite subgroups of the orthogonal group and has  previously been used to derive designs in Grassmann spaces in \cite{Bachoc:2004aa}, see \cite{Vale:2005aa} for lines in $\R^d$. We also verify that $p$-designs in complex and quaternionic projective spaces in the sense of \cite{Hoggar:1982aa} induce tight $p$-fusion frames. Another construction of tight $p$-fusion frames is presented that is reminiscent to constructions by concatenation in
coding theory.

The outline is as follows: In Section \ref{sec:fff}, we list the basic properties of fusion frames. We recall the simplex bound in Section \ref{section:equal simplex} and derive the generalized simplex bound in Section \ref{subsection:finite p frames}. An upper bound on the number of equiangular subspaces is given in Section \ref{subsec:maximal}. We introduce $p$-fusion frames in Section \ref{section:p fusion}. The relations between tight $p$-fusion frames and cubatures of strength $2p$ are investigated in Section \ref{sec:equal dim}. Constructions of tight $p$-fusion frames are discussed in Section \ref{sec:construction}. Section \ref{sec:p pot} contains a  lower bound on the $p$-fusion frame potential that does not require the subspaces to be equidimensional.

\section{Fusion frames}\label{sec:fff}
To introduce fusion frames, we follow \cite{Casazza:2008aa}, see also \cite{Casazza:2009aa,Kutyniok:2009aa}. Given a linear subspace $V\subset\R^d$, let $P_V$ denote the orthogonal projection onto $V$. The \emph{real Grassmann space} $\mathcal{G}_{k,d}$ is the space of all $k$-dimensional subspaces of $\R^d$. Moreover $\GG_d:=\cup_{k=1}^{d-1} \Gkd$ is the union of all Grassmann spaces. 
\begin{definition}
Let $\{V_j\}_{j=1}^n\subset \GG_d$ and let $\{\omega_j\}_{j=1}^n$ be a collection of positive weights. Then $\{(V_j,\omega_j)\}_{j=1}^n$ is called a \emph{fusion frame} if there are positive constants $A$ and $B$ such that
\begin{equation}\label{eq:fusion def}
A\|x\|^2 \leq \sum_{j=1}^n \omega_j\|P_{V_j}(x)\|^2 \leq B \|x\|^2, \text{ for all $x\in\R^d$.}
\end{equation}
If $A=B$, then $\{(V_j,\omega_j)\}_{j=1}^n$ is called a \emph{tight fusion frame}. 
\end{definition} 
In case that $A=B$ and the weights are all equal to $1$, we simply say that $\{V_j\}_{j=1}^n$ is a tight fusion frame. Moreover, we shall always assume that the subspaces $V_j$ are non trivial, i.e., that $V_j\neq \{0\}$ and $V_j\neq \R^d$. 

The standard inner product between self-adjoint operators $P$ and $Q$ is defined by $\langle P,Q\rangle:=\trace(PQ)$. It should be noted that if 
$x$ belongs to the unit sphere $S^{d-1}=\{x\in \R^d : \Vert x\Vert=1\}$, then $\|P_V(x)\|^2=\langle P_{x},P_V\rangle$, where $P_x$ stands for the orthogonal projection onto the line $\R x$. Thus, the fusion frame condition \eqref{eq:fusion def} is equivalent to 
\begin{equation}\label{eq:tight fusion def}
A\leq \sum_{j=1}^n \omega_j\langle P_{x},P_{V_j}\rangle \leq B, \text{ for all $x\in S^{d-1}$.}
\end{equation}


Let us recall the significance of fusion frames for signal reconstruction: any finite collection $\{V_j\}_{j=1}^n$ of linear subspaces in $\R^d$ with positive weights $\omega=\{\omega_j\}_{j=1}^n$ induces an \emph{analysis operator} 
\begin{equation}\label{eq:analysis fp}
F : \R^d \rightarrow  \big(\bigoplus_{j=1}^n V_j\big)_\omega,\quad x \mapsto \{P_{V_j}(x)\}_{j=1}^n,
\end{equation}
where $\big(\bigoplus_{j=1}^n V_j\big)_\omega$ is the space $\bigoplus_{j=1}^n V_j$ endowed with the inner product \\ $\langle \{f_j\}_{j=1}^n , \{g_j\}_{j=1}^n\rangle _\omega := \sum_{j=1}^n \omega_j\langle f_j , g_j\rangle $. Its adjoint is the \emph{synthesis operator} 
\begin{equation}\label{eq:synthesis fp}
F^* : \big(\bigoplus_{j=1}^n V_j\big)_\omega \rightarrow  \R^d, \quad \{f_j\}_{j=1}^n \mapsto \sum_{j=1}^n   \omega_j f_j,
\end{equation}
and the \emph{fusion frame operator} is defined by
\begin{equation}\label{eq:frame operator discrete fusion frame}
S:= F^* F : \R^d \rightarrow \R^d,\quad  x \mapsto  \sum_{j=1}^n  \omega_j P_{V_j}(x).
\end{equation}
If $\{V_j\}_{j=1}^n$ forms a   fusion frame, then $S$ is positive, self-adjoint, invertible, and induces the reconstruction formula 
$ 
x = \sum_{j=1}^n  \omega_j S^{-1} P_{V_j}(x),\text{ for all $x\in\R^d$},
$ 
cf.~\cite{Casazza:2008aa}. If the  fusion frame is tight, then $S=AI_d$ holds, and we obtain 
the appealing representation
\begin{equation}\label{eq:tight reconstr}
x = \frac{1}{A}\sum_{j=1}^n  \omega_j P_{V_j}(x),\text{ for all $x\in\R^d$.}
\end{equation}

\section{The simplex bound and equiangular fusion frames}
Our goal in this section is to give lower bounds on $\max_{i\neq j} (\langle P_{V_i},P_{V_j}\rangle)$. 
We start to review the known results in the case of subspaces of equal dimension.
\subsection{The simplex bound for subspaces of equal dimension}\label{section:equal simplex}
The \emph{chordal distance} $d_c(V,W)$, for $V, W\in\mathcal{G}_{k,d}$,   was introduced in \cite{Conway:1996aa} and is defined by
\begin{equation}\label{eq:sin cos}
d^2_c(V,W) =\frac{1}{2}\|P_V-P_W\|^2_{\mathcal{F}}= k-\sum_{i=1}^k \cos^2(\theta_i(V,W))= k-\langle P_{V} ,P_{W}\rangle,
\end{equation}
where $\theta_1,\ldots,\theta_k$ are the \emph{principal angles} between $V$ and $W$, cf.~\cite{Golub:1996fk}. If $\{V_j\}_{j=1}^n\subset \mathcal{G}_{k,d}$, then the \emph{simplex bound} as derived in \cite{Conway:1996aa} yields
\begin{equation}\label{eq:simplex}
\min_{i\neq j} d_c^2(V_i,V_j) \leq \frac{k(d-k)}{d}\frac{n}{n-1},
\end{equation}
and equality requires $n\leq \binom{d+1}{2}$. Of course, in view of \eqref{eq:sin cos}, the above simplex bound is a lower bound for $\max_{i\neq j} (\langle P_{V_i},P_{V_j}\rangle)$.
The following  result is proven in \cite{Kutyniok:2009aa}:
\begin{theorem}[\cite{Kutyniok:2009aa}]\label{th:gittaSimplexBound}
If $\{V_j\}_{j=1}^n\subset \mathcal{G}_{k,d}$ is an equidistance fusion frame, i.e. if $d_c(V_i,V_j)$ is independent of $i\neq j$, then it is tight if and only if it satisfies the simplex bound \eqref{eq:simplex} with equality.
\end{theorem}
We shall extend the simplex bound \eqref{eq:simplex} and the above theorem to collections of weighted subspaces that do not all have the same dimension. 
\subsection{The simplex bound for subspaces of arbitrary dimension}\label{subsection:finite p frames}
When the subspaces do not have the same dimension, we replace the notion of subspaces being equidistant with the notion of 
\emph{equiangular subspaces}:
\begin{definition}
Let $\{V_j\}_{j=1}^n$ be a collection in $\mathcal{G}_d$ and let $\{\omega_j\}_{j=1}^n$ be positive weights. We then call both,  $\{(V_j,\omega_j)\}_{j=1}^n$ and $\{V_j\}_{j=1}^n$, \emph{equiangular} if  $\langle P_{V_i},P_{V_j}\rangle$ does not depend on $i\neq j$.  
\end{definition}
If all the subspaces $V_j$ are of dimension $1$, then our definition of $\{V_j\}_{j=1}^n$  being equiangular coincides with the classical notion of equiangular lines. 
If all subspaces $V_j$ are of dimension $k$, being equiangular amounts to being equidistant with respect to the chordal distance.
However, we remark that being equiangular in our sense does not mean that the $k$-tuples of principal angles between the pairs $(V_i,V_j)$ are the same (unless $k=1$).

The proof of the simplex bound \eqref{eq:simplex} in \cite{Conway:1996aa} heavily relies on the embedding of $\mathcal{G}_{k,d}$ into a higher dimensional sphere. For subspaces that are not equidimensional, we cannot use this embedding. Instead, we use the \emph{$p$-fusion frame potential}.
\begin{definition} \label{def:ffp}
The \emph{$p$-fusion frame potential} of the collection of weighted subspaces $\{(V_j,\omega_j)\}_{j=1}^n$ is defined
for $1\leq p<\infty$  by:
\begin{equation}\label{pFFP}
\FFP(\{(V_j,\omega_j)\}_{j=1}^n,p):=\sum_{i,j=1}^n \omega_i \omega_j\langle P_{V_i}, P_{V_j}\rangle^p.
\end{equation}
\end{definition}
Note that the $1$-fusion frame potential $\FFP(\{(V_j,\omega_j),1\}_{j=1}^n)=\trace(S^2)$, where $S$ is the fusion frame operator, has already been considered in \cite{Casazza:2009aa,Massey:2010fk}.

We can now derive a new weighted simplex bound for collections of subspaces that are not necessarily equidimensional. 
\begin{theorem}[Generalized Simplex Bound]\label{prop:generalized simplex bound}
Given positive weights $\{\omega_j\}_{j=1}^n$, if $\{V_j\}_{j=1}^n\subset \mathcal{G}_{d}$ and $m=\sum_{j=1}^n \omega_j\dim(V_j)$, then the following points hold:
\begin{itemize}
\item[\textnormal{1)}] For $1\leq p<\infty$,
\begin{equation}\label{eq:sharp equi distant}
\hspace{-.3cm}
\FFP(\{(V_j,\omega_j)\}_{j=1}^{n},p)\geq \frac{\big(\frac{m^2}{d}-\sum_{j=1}^n\omega_j^2\dim(V_j)\big)^p}{(\sum_{i\neq j}\omega_i\omega_j)^{p-1}}+\sum_{j=1}^n \omega^{2}_j\dim(V_j)^p \!.
\end{equation}
If $p=1$, then equality holds if and only if $\{(V_j,\omega_j)\}_{j=1}^n$ is a tight fusion frame. If $1<p<\infty$, then equality holds if and only if $\{(V_j,\omega_j)\}_{j=1}^n$ is an equiangular tight fusion frame. 
\item[\textnormal{2)}] 
\begin{equation}\label{eq:simplex for trace}
\max_{i\neq j} \langle P_{V_i},P_{V_j} \rangle \geq  \frac{\frac{m^2}{d}-\sum_{j=1}^n\omega_j^2\dim(V_j)}{\sum_{i\neq j} \omega_i\omega_j}.
\end{equation}
\end{itemize}
Equality holds if and only if $\{(V_j,\omega_j)\}_{j=1}^n$ is an equiangular tight fusion frame.
\end{theorem}
\begin{proof}
Part 1) for $p=1$ has already been derived in \cite{Casazza:2009aa,Massey:2010fk}. For $1<p<\infty$, we take the $p$-th root and only consider the terms $i\neq j$. We then see that \eqref{eq:sharp equi distant} is equivalent to 
\begin{equation*}
\|(\omega^{1/p}_i\omega_j^{1/p}\langle  P_{V_i},P_{V_j} \rangle)_{i\neq j}\|_{\ell^p} \geq \big(\frac{m^2}{d}-\sum_{j=1}^n\omega_j^2\dim(V_j)\big) \big(\sum_{i\neq j} \omega_i\omega_j\big)^{1/p-1},
\end{equation*}
so that \eqref{eq:simplex for trace} complements the estimate on $\|(\omega_i^{1/p}\omega_j^{1/p}\langle  P_{V_i},P_{V_j} \rangle)_{i\neq j}\|_{\ell^p}$ in \eqref{eq:sharp equi distant} for $p=\infty$.  

By applying the H\"older inequality, we obtain, for $1<p\leq \infty$ and $1=\frac{1}{p}+\frac{1}{q}$,
\begin{align}
  \|(\omega^{\frac{1}{p}}_i\omega_j^{\frac{1}{p}}\langle  P_{V_i},P_{V_j} \rangle)_{i\neq j}\|_{\ell^p} \|(\omega^{\frac{1}{q}}_i \omega^{\frac{1}{q}}_j)_{i\neq j}\|_{\ell^q}
&   \ge \sum_{i \ne j} \omega_i\omega_j\langle P_{V_i}, P_{V_j} \rangle \label{eq:first inequality in proof} \\
& = \trace\big(\big(\sum_{i=1}^n \omega_iP_{V_i}\big)^2 \big) -\sum_j w_j^2 \dim V_j\nonumber 
\intertext{If $\{\lambda_k\}_{k=1}^d$ are the eigenvalues of $S=\sum_{i=1}^n \omega_iP_{V_i}$, then we further obtain}
\trace\big(\big(\sum_{i=1}^n \omega_iP_{V_i}\big)^2 \big) -\sum_j w_j^2 \dim V_j& = \sum_{k=1}^d\lambda_k^2 -\sum_j w_j^2 \dim V_j\nonumber \\
&\geq \frac{1}{d}\big(\sum_{k=1}^d \lambda_k\big)^2- \sum_j w_j^2 \dim V_j.\nonumber
\end{align}
In the last step we have used the Cauchy-Schwartz inequality for $(\lambda_k)_{k=1}^d$ and the constant sequence. 
The inequality \eqref{eq:first inequality in proof} turns into an equality if and only if $\{V_j\}_{j=1}^n$ are equiangular. The Cauchy-Schwartz inequality turns into an equality if and only if $S$ is a multiple of the identity.
\end{proof}
By applying \eqref{eq:sin cos}, we observe that \eqref{eq:simplex for trace} is equivalent to the simplex bound \eqref{eq:simplex} if all the subspaces have the same dimension and the weights are constant. 

\begin{corollary}\label{corol:1/n}
If $\{V_j\}_{j=1}^n\subset \mathcal{G}_{k,d}$ and $\{\omega_j\}_{j=1}^n$ are positive weights, 
then $\{(V_j,\omega_j)\}_{j=1}^n$ is an equiangular tight fusion frame if and only if the weights are constant 
and $\langle P_{V_i},P_{V_j}\rangle = \frac{k(nk-d)}{(n-1)d}$, for all $i\neq j$.
\end{corollary}
\begin{proof}
Without loss of generality, we can assume that $\sum_{j=1}^n \omega_j=1$. For $\dim(V_j)=k$, $j=1,\ldots,n$, the right-hand side of \eqref{eq:simplex for trace} equals $k(\frac{\frac{k}{d}-1}{1-\sum_{j=1}^n \omega^2_j}+1)$ and is maximized 
if and only if $\omega_j=1/n$, $j=1,\ldots,n$. By applying Theorem \ref{prop:generalized simplex bound}, we can conclude the proof.  
\end{proof}

\subsection{The maximal number of equiangular subspaces}\label{subsec:maximal}
To match the generalized simplex bound of Theorem \ref{prop:generalized simplex bound} with equality, the subspaces need to be equiangular. It is natural to ask how large a collection of equiangular subspaces can be.
The classical Gerzon upper bound  $n\leq \binom{d+1}{2}$ for equiangular lines \cite{LemmensSeidel:1973aa} was extended to equiangular 
subspaces of equal dimension $k$ in \cite[Theorem 3.6]{Bachoc:2004aa}. In the present section, we prove that this upper bound extends further to equiangular subspaces of arbitrary dimensions. 
\begin{theorem}\label{theorem:equiangular general}
If $\{V_j\}_{j=1}^n$ is a collection of equiangular pairwise distinct subspaces in $\mathcal{G}_{d}$, then $n\leq \binom{d+1}{2}$. 
\end{theorem}
\begin{proof}
Let $ \alpha=\langle P_{V_i},P_{V_j}\rangle$, for all $i\neq j$. We split the proof into two cases. 

Case 1) Suppose that there exists $i$ such that $\alpha=\dim(V_{i})$. Without loss of generality, we assume that $i=1$. A short computation yields that $ \langle P_{V_{1}},P_{V}\rangle \leq \dim(V_{1})$, for all $V\in\mathcal{G}_d$, and equality holds if and only if $V_{1}$ is contained in $V$. Thus, we have $V_1\subset V_j$, for all $j=2,\ldots,n$. Let $W_j$ be the orthogonal complement of $V_{1}$ in $V_j$, i.e., $V_j=V_{1}\oplus W_j$, for all $j=2,\ldots,n$. It can be checked that the equiangularity implies that the collection $\{W_j\}_{j=2}^n$ is pairwise orthogonal. Since $\{V_j\}_{j=1}^n$ are pairwise distinct, none of the $\{W_j\}_{j=2}^n$ can be zero. Thus, $n-1\leq d$ must hold, which implies $n\leq \binom{d+1}{2}$, for $d\geq 2$.

Case 2) We can now suppose that $\alpha< \dim(V_i)=:k_i$, for all $i=1,\ldots,n$. Let us define the matrix $\Gram:=(\langle P_{V_{i}},P_{V_j}\rangle )_{i,j}\in\R^{n\times n}$, so that
\begin{equation*}
\Gram = \begin{pmatrix}
k_1 & \alpha &\cdots &\alpha\\
\alpha &\ddots & & \vdots\\
\vdots & & \ddots &\alpha\\
\alpha & \cdots &\alpha& k_n
\end{pmatrix}.
\end{equation*}
We can check by induction and elimination that $\Gram$ has full rank. Therefore, $\{P_{V_j}\}_{j=1}^n$ is linearly independent. Since the real vector space of self-adjoint matrices is $\binom{d+1}{2}$-dimensional, we must have $n\leq \binom{d+1}{2}$. 
\end{proof}

The following examples form 
equiangular subspaces:
\begin{example}\label{example:equiangular}
\begin{itemize}
\item[1)] A collection of $10$ two-dimensional subspaces of $\R^4$ was constructed in \cite{Conway:1996aa} that match the simplex bound. 
\item[2)] Let $d$ be a prime which is either $3$ or congruent to $-1$ modulo $8$. A collection of $\binom{d+1}{2}$ subspaces in $\mathcal{G}_{\frac{d-1}{2},d}$ satisfying the simplex bound was constructed in \cite{Calderbank:1999uq}.
\item[3)] In \cite{Creignou:2008aa}, codes in Grassmann spaces were constructed from $2$-transitive  groups. 
By construction, these codes are equiangular.
\end{itemize}
\end{example}

\section{Tight $p$-fusion frames}\label{section:p fusion}
The notion of (tight) fusion frames generalizes in a natural way when squares are replaced by $2p$-powers for $p$ a positive integer:
\begin{definition}
Let $\{V_j\}_{j=1}^n$ be a collection of linear subspaces in $\R^d$ and let $\{\omega_j\}_{j=1}^n$ be a collection of positive weights. Then $\{(V_j,\omega_j)\}_{j=1}^n$ is called a \emph{$p$-fusion frame} if there exist constants $A,B>0$ such that
\begin{equation}\label{eq:p fusion def}
A\|x\|^{2p}\leq \sum_{j=1}^n \omega_j\|P_{V_j}(x)\|^{2p} \leq B\|x\|^{2p}, \text{ for all $x\in\R^d$.}
\end{equation}
If the weights are all equal to $1$, then we suppress them in our notation and simply write $\{V_j\}_{j=1}^n$ for the $p$-fusion frame. If $A=B$, then $\{(V_j,\omega_j)\}_{j=1}^n$ is called a \emph{tight $p$-fusion frame}. If, in addition, all the subspaces are one-dimensional, then $\{(V_j,\omega_j)\}_{j=1}^n$ is simply called a \emph{tight $p$-frame}. 
\end{definition} 
Of course, tight $1$-fusion frames are tight fusion frames. Also, it is clear from the definition that the union of tight $p$-fusion frames is again a tight $p$-fusion frame.  

Now we show that, for a tight $p$-fusion frame, the value of $A=B$ is uniquely determined. The real Grassmann space $\mathcal{G}_{k,d}$ 
is endowed with the transitive action of the real orthogonal group $O(\R^d)$. The Haar measure on $O(\R^d)$ induces a measure $\sigma_k$ on the Grassmann space $\mathcal{G}_{k,d}$, that we assume to be normalized, i.e. $\sigma_k(\mathcal{G}_{k,d})=1$. Because these measures are 
$O(\R^d)$-invariant, the integral $\int_{\Gkd} \langle P_x,P_V\rangle^{p}d\sigma_k(V)$ does not depend on the choice of $x\in \R^d$, and similarly,  $\int_{\God} \langle P_x,P_V\rangle^{p}d\sigma_1(\R x)$  is independent of $V\in \Gkd$. Therefore, we can define the value $T_{1,k,d}(p)$ by 
\begin{equation}\label{eq:T1kd}
T_{1,k,d}(p):=\int_{\God}\int_{\Gkd} \langle P_x,P_V\rangle^{p}d\sigma_k(V) d\sigma_1(x)=\int_{\Gkd} \langle P_x,P_V\rangle^{p}d\sigma_k(V)=\int_{\God} \langle P_x,P_V\rangle^{p}d\sigma_1(x).
\end{equation}
The defining property of a $p$-fusion frame can be rephrased in the following way:
\begin{equation}\label{eq:p fusion def2}
A\leq \sum_{j=1}^n \omega_j\langle P_x, P_{V_j}\rangle^{p} \leq B, \text{ for all $x\in S^{d-1}$.}
\end{equation}
Integrating \eqref{eq:p fusion def2} over $\R x\in \God$ and using \eqref{eq:T1kd} lead to
\begin{equation}\label{eq:specification of A}
A\leq \sum_{k=1}^{d-1} m_k T_{1,k,d}(p)\leq B,
\end{equation}
where $m_k = \sum_{\dim(V_j)=k} \omega_j$. Equality holds for tight $p$-fusion frames. Since $T_{1,k,d}(1)=\frac{k}{d}$, cf.~\cite{James:1974aa}, the frame bounds of a fusion frame satisfy $A\leq \frac{m}{d}\leq B$, where $m=\sum_{j=1}^n \omega_j\dim(V_j)$.

It should also be mentioned that reweighting of a tight $p$-fusion frame leads to tight $p'$-fusion frames for the entire range $1\leq p'\leq p$:
\begin{theorem}\label{th:tight p for p'}
Let $p\geq 2$. If $\{(V_j,\omega_j)\}_{j=1}^n$ is a tight $p$-fusion frame, then $\{(V_j,\tilde{\omega}_j)\}_{j=1}^n$ is a tight $(p-1)$-fusion frame, where $\tilde{\omega}_j=\omega_j(p-1+\dim(V_j)/2)$.  
\end{theorem}
\begin{proof}
We introduce the Laplace operator $\Delta=\sum_{i=1}^ d \frac{\partial^2}{\partial x_i^2}$. In spherical coordinates, for $x\neq 0$, we use the parametrization $x=r\varphi$, for $r >0$ and $\varphi\in S^{d-1}$, so that the function $f(x)=\|x\|^{2p}$ is constant in $\varphi$. Thus, we have $\Delta f=r^{1-d}\partial_r(r^{d-1}\partial_r f)$, which yields 
\begin{equation*}
\Delta \big(\| x \|^{2p}\big)= 4p(p-1+d/2) \| x \|^{2(p-1)}.
\end{equation*}
More generally, for a subspace $V$, we obtain
\begin{equation*}
\Delta \big(\| P_V( x) \|^{2p}\big) = 4p(p-1+\dim(V)/2) \| P_V(x) \|^{2(p-1)}.
\end{equation*}
Applying $\Delta$ to both sides of the identity $\sum_{j=1}^n \omega_j\|P_{V_j}(x)\|^{2p} = A \| x \|^{2p}$, we obtain 
\begin{equation*}
\sum_{j=1}^n \omega_j (p-1+\dim(V_j)/2)\|P_{V_j}(x)\|^{2(p-1)} = A (p-1+d/2) \| x \|^{2(p-1)},
\end{equation*}
 proving that $\{(V_j,\omega_j (p-1+\dim(V_j)/2))\}_{j=1}^n$ is a $(p-1)$-tight fusion frame. 
 \end{proof}
\begin{remark}
Iteration of Theorem \ref{th:tight p for p'} yields that if $\{(V_j,\omega_j)\}_{j=1}^n$ is a tight $p$-fusion frame, then $\{(V_j,\omega'_j)\}_{j=1}^n$ is a tight fusion frame, where $\omega'_j=\omega_j\prod_{l=1}^{p-1}(l+\dim(V_j)/2)$.
\end{remark}
\section{Equidimensional tight $p$-fusion frames, cubature formulas and the $p$-fusion frame potential}\label{sec:equal dim}
In this section,  we assume that the subspaces $V_j$ have the same dimension $k$.  Using tools from  harmonic analysis on the Grassmann manifold $\Gkd$, 
the tight $p$-fusion frames can be characterized in terms of the \emph{principal angles} of the pairs of subspaces. The same holds for minimizers of the $p$-fusion frame potential, where we minimize over all collections of $k$-dimensional subspaces whose weights add up to one. We shall recognize in these minimizers the \emph{cubatures} for the Grassmann space, also called \emph{Grassmann designs} in the case of constant weights. 
It will turn out that the minimizers of the $p$-fusion frame potential are tight $p$-fusion frames, while the converse holds only in the cases $p=1$ or $k=1$.

The use of harmonic analysis, namely the irreducible decomposition of $L^2$ and the associated zonal spherical functions, is standard in the study of designs in homogeneous spaces. The unit sphere of Euclidean space  \cite{Delsarte:1977aa, Venkov:2001aa}
served as a model for many other spaces \cite{Neumaier:1981aa, Hoggar:1982aa, Bachoc:2002aa}. We refer to \cite{Harpe:2005aa} for a general framework
 for cubature formulas in polynomial spaces and to \cite{Bachoc:2002aa}  for the notion of designs in Grassmann spaces (see also \cite{Bachoc:2004aa, Bachoc:2005aa}). 

\subsection{A closed formula for the tight $p$-fusion frame bound}
The next proposition shows that, after possibly a change from  $\{V_j\}$ to $\{V_j^\perp\}$, the condition $k\leq d/2$  can be fulfilled. The  assumption that $k\leq d/2$ will be conveniently followed in the remaining of this section.
\begin{proposition}\label{prop equal dim}
If $\{(V_j,\omega_j)\}_{j=1}^n$ is a $p$-tight fusion frame of equal dimension $k$, then:
\begin{enumerate}
\item[(1)] $\{(V_j,\omega_j)\}_{j=1}^n$ is a $p'$-tight fusion frame for all $1\leq p'\leq p$.
\item[(2)] $\{(V_j^\perp,\omega_j)\}_{j=1}^n$ is also a $p$-tight fusion frame.
\end{enumerate}
\end{proposition}
\begin{proof}
Part (1) follows from Theorem \ref{th:tight p for p'} by putting $(p-1+k/2)^{-1}$ into the fusion frame constant. For (2), we observe that $\| x \|^2=\|P_{V_j}(x)\|^2+\|P_{{V_j}^\perp}(x)\|^2$, so
\begin{align*}
\sum_{j=1}^n \omega_j\|P_{V_j^\perp}(x)\|^{2p} &= \sum_{j=1}^n \omega_j(\| x \|^2-\|P_{V_j}(x)\|^2)^{p} \\
&=\sum_{k=0}^p  (-1)^k \binom{p}{k} \| x\|^{2(p-k)} \sum_{j=1}^n \omega_j\|P_{V_j}(x)\|^{2k}\\
&=\sum_{k=0}^p  (-1)^k \binom{p}{k} \| x\|^{2(p-k)} A_k \| x\|^{2k} \\
&= \Big(\sum_{k=0}^p  (-1)^k \binom{p}{k} A_k \Big)\| x\|^{2p}, 
\end{align*}
where the second last equality, follows from the property that $\{(V_j,\omega_j)\}_{j=1}^n$ is a $k$-tight fusion frame for all $1\leq k\leq p$, and insures  the existence of some constants $A_k>0$ such that $\sum_{j=1}^n \omega_j\|P_{V_j}(x)\|^{2k}= A_k \| x\|^{2k}$. Since $V_1^\perp$ is not empty, $\sum_{k=0}^p  (-1)^k \binom{p}{k} A_k>0$, so that  $\{(V^\perp_j,\omega_j)\}_{j=1}^n$ is a tight $p$-fusion frame.
\end{proof}

\begin{remark}
It follows from \eqref{eq:specification of A} that
the constant $A_p$ in the characteristic property of tight $p$-fusion frames $\sum_{j=1}^n \omega_j \|P_{V_j}(x)\|^{2p} = A_p \| x \|^{2p}$ equals $A_p=T_{1,k,d}(p) \sum_{j=1}^n \omega_j$. Applying the Laplace operator $p$ times leads to another, more explicit, formula:
\begin{equation}\label{eq:new bound formula}
A_p= \frac{(k/2)_p}{(d/2)_p}\sum_{j=1}^n\omega_j,
\end{equation}
where we employ the standard notation $(a)_p=a(a+1)\cdots (a+p-1)$.
\end{remark}

\subsection{Characterization of tight $p$-fusion frames by means of principal angles}
We now review the irreducible decomposition of the Hilbert space
$L^2(\Gkd)$ of complex valued functions of integrable squared module,
under the action of the orthogonal group $O(\R^d)$.  The standard
inner product on $L^2(\Gkd)$ is denoted $\langle f,g\rangle$.  Let
$V_d^{\mu}$ denote the complex irreducible representation of $O(\R^d)$ canonically associated to the partition
$\mu=\mu_1\geq \dots\geq \mu_d\geq 0$ (see \cite{Goodman:1998aa}). For such a partition $\mu=(\mu_1,\dots,\mu_d)$ with parts $\mu_i$, its \emph{degree} $\deg(\mu)$  is the sum of its parts and its \emph{length}  $l(\mu)$ is the number of its non zero parts.
We usually omit the parts equal to $0$ in the notation of a partition. For example, $V_d^{(0)}$ is the trivial representation, and $V_d^{(\ell)}$ is the representation afforded by the homogeneous harmonic polynomials of degree $\ell$ (i.e.~the kernel of the Laplace operator). 
Then we have:
\begin{equation}\label{dec L2}
L^2(\Gkd)=\bigoplus_{ l(\mu)\leq k} H_{k,d}^{2\mu}, \quad\text{ where } H_{k,d}^{2\mu} \simeq V_d^{2\mu}.
\end{equation}
Here $2\mu=(2\mu_1,\dots,2\mu_d)$ runs over the partitions with even parts. 
The subspace 
\begin{equation}
\Pol_{\leq 2p} (\Gkd):=\bigoplus_{\ l(\mu)\leq k,\ \deg(\mu)\leq p} H_{k,d}^{2\mu}
\end{equation}
coincides with the space of polynomial functions on $\Gkd$ of degree bounded by $2p$. We also introduce the subspace
\begin{equation}
\Pol^1_{\leq 2p} (\Gkd):=\bigoplus_{\ \ell \leq p} H_{k,d}^{(2\ell)}\subset \Pol_{\leq 2p} (\Gkd),
\end{equation}
so that the orthogonal complement of $\Pol^1_{\leq 2p} (\Gkd)$ in $\Pol_{\leq 2p} (\Gkd)$ is the direct sum of all $ H_{k,d}^{2\mu}$, such that $2\leq l(\mu)\leq k$ and $\deg(\mu)\leq p$.

We recall that $k$ \emph{principal angles} $(\theta_1,\dots, \theta_k)\in [0,\pi/2]^k$  are associated  to a pair of subspaces $(V,W)$ of $\R^d$ with $d/2\geq \dim(V)=l\geq \dim(W)=k$. We let $y_i:=\cos^2(\theta_i)$.  Then, $y_1,\dots,y_k$ are exactly the non zero eigenvalues of the operator $P_VP_W$. In particular, we observe that $y_1+\dots+y_k=\langle P_V,P_W\rangle$.  The set $\{y_1,\dots,y_k\}$ uniquely characterizes the orbit of the pair $(V,W)$ under the action of $O(\R^d)$.

To every subspace $H_{k,d}^{2\mu}$ is associated a polynomial $P_{\mu}(y_1,\dots,y_k)$ which is symmetric in the variables $y_i$, of degree equal to $\deg(\mu)$,  satisfying $P_{\mu}(1,\dots,1)=1$, and such that
$V\mapsto P_{\mu}(y_1(V,W),\dots,y_k(V,W))$ belongs to $H_{k,d}^{2\mu}$. In fact, these two last properties uniquely determine $P_{\mu}$.  For example, $P_{(0)}=1$ and $P_{(1)}=(y_1+\dots+y_k)-k^2/d$ up to a multiplicative constant. These polynomials are called the \emph{zonal spherical polynomials} of the Grassmann manifold. They were calculated in \cite{James:1974aa}, where it is shown that they belong to the family of multivariate Jacobi polynomials. They do depend on the parameters $k$ and $d$, although those parameters are not involved in our notation, see also \cite{Bachoc:2006aa}.

Moreover, the functions $(V,W)\mapsto P_{\mu}(y_1(V,W),\dots, y_k(V,W))$ are \emph{positive definite functions} on $\Gkd$, meaning that, for all $n\geq 1$ and all $\{V_j\}_{j=1}^n\subset \Gkd$, the matrix $(P_{\mu}(y_1(V_i,V_j),\dots, y_k(V_i,V_j)))_{1\leq i,j\leq n}$ is
positive semidefinite.  As a consequence, we have:
\begin{equation}\label{psd}
\sum_{i,j=1}^n \omega_i \omega_j P_{\mu}(y_1(V_i,V_j),\dots, y_k(V_i,V_j))\geq 0, \quad \text{ for all } \{V_j\}_{j=1}^n\subset \Gkd.
\end{equation}
Taking $\mu=(1)$, the inequality \eqref{psd} becomes 
\begin{equation*}
\FFP(\{(V_j,\omega_j)\}_{j=1}^n,1)\geq \frac{1}{d}\big(\sum_{j=1}^n \omega_j k\big)^2, 
\end{equation*}
so we already see  here a connection with Theorem \ref{prop:generalized simplex bound}. Now we are in the position to characterize the tight $p$-fusion frames.

\begin{theorem}\label{Th weak designs}
The following properties are equivalent for $\{(V_j,\omega_j)\}_{j=1}^n$, where $\{V_j\}_{j=1}^n\subset \mathcal{G}_{k,d}$ and $\sum_{j=1}^n \omega_j=1$:
\begin{enumerate}
\item[(1)] $\{(V_j,\omega_j)\}_{j=1}^n$ is a tight $p$-fusion frame.
\item[(2)] For all $f\in \Pol_{\leq 2p} ^1(\Gkd)$, 
\begin{equation}\label{weak design}
\int_{\Gkd} f(V)d\sigma_k(V)= \sum_{j=1}^n \omega_j f(V_j).
\end{equation}
\item[(3)] For $1\leq \ell\leq p$, for all $f\in H_{k,d}^{(2\ell)}$, $\sum_{j=1}^n \omega_j f(V_j)=0$.
\item[(4)] For $1\leq \ell\leq p$, $\sum_{i,j=1}^n\omega_i \omega_j P_{(\ell)}(y_1(V_i,V_j),\dots, y_k(V_i,V_j))=0$.
\end{enumerate}
\end{theorem} 

\begin{proof} The proof of the equivalence of (2), (3), and (4) is similar to the proof of  \cite[Proposition 4.2]{Bachoc:2002aa}, so we skip it. 
Let, for $x\in \R^d$, $s_x^p(V):=\langle P_x,P_V \rangle^p$. Clearly $s_x^p\in \Pol_{\leq 2p} (\Gkd)$. We claim that  $s_x^p\in \Pol^1_{\leq 2p} (\Gkd)$. 
Let $f\in H_{k,d}^{2\mu}$ with $l(\mu)\geq 2$; we want to prove that $\langle s_x^p, f\rangle =0$. Indeed, the application that sends $f\in H_{k,d}^{2\mu}$ to $\R x\mapsto \langle s_x^p, f\rangle\in L^2(\God)$ is $O(\R^d)$-equivariant. 
Because $L^2(\God)\simeq \bigoplus_{\ell\geq 0} V_d^{(2\ell)}$ does not contain the representation $V_d^{2\mu}\simeq H_{k,d}^{2\mu}$, by Schur's lemma, this application has to be identically zero.

Let $\Sigma$ denote the subspace of $L^2(\Gkd)$ spanned by the functions $s_x^p$ when $x$ runs in $\R^d$. We observe that $\Sigma$ is invariant under the action of the orthogonal group. We have just proved that $\Sigma\subset \Pol^1_{\leq 2p} (\Gkd)$, so  (2) implies (1). For the converse implication, we need to prove that $\Sigma= \Pol^1_{\leq 2p} (\Gkd)$. Because  $\Pol^1_{\leq 2p} (\Gkd)$ is the direct sum of the irreducible and pairwise non isomorphic $O(\R^d)$-subspaces $H_{k,d}^{(2\ell)}$ for $0\leq \ell \leq p$, either 
$H_{k,d}^{(2\ell)}\subset \Sigma$, or $H_{k,d}^{(2\ell)}$ and $\Sigma$ are orthogonal.  In order to rule out this second possibility, we call for another sequence of polynomials denoted $P_{(\ell)}^{1,k}(y_1)$.
These polynomials are orthogonal for the measure $y_1^{(k-2)/2}(1-y_1)^{(d-2-k)/2}dy_1$ over the interval $[0,1]$, which is the measure induced on $y_1(x,V)$ by the measures on the Grassmann spaces, and are normalized by the property 
$P_{(\ell)}^{1,k}(1)=1$. Here $y_1(x,V)$ stands for $y_1(\R x,V)=\langle P_x,P_V\rangle$. These polynomials are characterized (up to a multiplicative factor)
by the property 
that $\R x\mapsto P_{(\ell)}^{1,k}(y_1(x,V))$ belongs to $H_{1,d}^{(2\ell)}$ and $V\mapsto P_{(\ell)}^{1,k}(y_1(x,V))$ belongs to $H_{k,d}^{(2\ell)}$ (see \cite{James:1974aa}). So, it is enough to prove that
\begin{equation*}
\int_{\Gkd} \langle P_x,P_V\rangle^p P_{(\ell)}^{1,k}(y_1(x,V)) d\sigma_k(V) \neq 0 \quad \text{for } 0\leq \ell \leq p
\end{equation*}
or equivalently that 
\begin{equation}\label{non neg}
\int_0^1 y_1^p P_{(\ell)}^{1,k}(y_1)  y_1^{(k-2)/2}(1-y_1)^{(d-2-k)/2}dy_1 \neq 0 \quad \text{for } 0\leq \ell \leq p.
\end{equation}
In fact, we can prove by induction on $p$ that the integral in \eqref{non neg} is positive. In the inductive step, we let $y_1^p P_{(\ell)}^{1,k}(y_1)=y_1^{p-1}(y_1P_{(\ell)}^{1,k}(y_1))$, and we replace 
$y_1P_{(\ell)}^{1,k}(y_1)$ by 
\begin{equation*}
y_1P_{(\ell)}^{1,k}(y_1) = a_{\ell} P_{(\ell+1))}^{1,k}(y_1) +b_{\ell}P_{(\ell)}^{1,k}(y_1) +c_{\ell}P_{((\ell-1))}^{1,k}(y_1), 
\end{equation*}
where the coefficients $a_{\ell}$, $b_{\ell}$ and $c_{\ell}$ can be computed from the coefficients in the three terms relation of the classical Jacobi polynomials in one variable 
(see \cite{Szego:1939aa}). It turns out  fortunately that $a_{\ell}$, $b_{\ell}$ and $c_{\ell}$ are  positive numbers. 
\end{proof}

\begin{remark}\label{remark:before curbature rules}
\begin{enumerate}
\item Theorem \ref{Th weak designs}(4) is the characterization of tight $p$-fusion frames we were aiming at, involving only the principal angles of the pairs $(V_i,V_j)$. 
\item The characteristic property (2) is reminiscent to so-called \emph{cubature formulas}. If the most classical setting for cubature formulas is numerical integration of polynomial  
functions on an interval of the real numbers,  they have also been extensively studied over other spaces such as the unit sphere of Euclidean space, although not over Grassmann spaces to
our knowledge. In \cite{Harpe:2005aa} a general framework
is provided for cubature formulas in polynomial spaces. Following \cite[Definition 1.3]{Harpe:2005aa}, a sequence of functional spaces $\FFF^{(p)}$ is said to be polynomial  if $\FFF^{(0)}=\C$  and  if $\FFF^{(p)}$ is generated by the products of elements of $\FFF^{(1)}$ and of $\FFF^{(p-1)}$.  It should be noted that the spaces $\Pol_{\leq 2p} ^1(\Gkd)$ are not ``polynomial spaces'' in this sense when $k>1$. Indeed, the products $f_1f_2$, where $f_i\in \Pol_{\leq 2} ^1(\Gkd)$, span $\Pol_{\leq 4} (\Gkd)$, which is larger than $\Pol_{\leq 4} ^1(\Gkd)$ when $k\geq 2$. So it is more adequate to define
cubature formulas for the elements of $\Pol_{\leq 2p}(\Gkd)$, which are polynomial.
\end{enumerate}
\end{remark}

\subsection{Cubature formulas as minimizers of the $p$-fusion frame potential}
In this section, we define cubature formulas on the Grassmann space and discuss their relations to the $p$-fusion frame potential.
\begin{definition}
Let $\{V_j\}_{j=1}^n$ be a finite subset of $\Gkd$ and let $\{\omega_j\}_{j=1}^n$ be a collection of positive weights, with $\sum_{j=1}^n \omega_j=1$. Then $\{(V_j,\omega_j)\}_{j=1}^n$ is called a \emph{cubature formula of strength $2p$} (or for short a cubature of strength $2p$)  if:
\begin{equation}\label{def cubature formula}
\int_{\Gkd} f(V)d\sigma_k(V)= \sum_{j=1}^n \omega_j f(V_j) \quad \text{ for all } f\in \Pol_{\leq 2p} (\Gkd).
\end{equation}
We say that $\{V_j\}_{j=1}^n$ is a  \emph{design of strength $2p$} or  a $2p$-design if $\{(V_j,1/n)\}_{j=1}^n$ is a cubature of strength $2p$.
\end{definition}

\begin{remark}
If $n=\binom{d+1}{2}$ holds in Theorem \ref{theorem:equiangular general} and all subspaces have the same dimension, then it follows from 
 \cite[Theorem 3.6]{Bachoc:2004aa} that $\{V_j\}_{j=1}^n$ is a $4$-design. 
\end{remark}

Cubatures  can be characterized in a similar way as tight $p$-fusion frames  with the help of the zonal spherical polynomials of the Grassmann manifold, and 
they also match lower bounds on the weighted $p$-potential. These results extend straightforwardly similar characterizations of designs on the unit sphere and in Grassmann spaces, 
see \cite{Bachoc:2005aa,Bachoc:2002aa,Delsarte:1977aa,Venkov:2001aa}. For preparation and extending \eqref{eq:T1kd}, we define, for $1\leq k,l\leq d-1$, 
\begin{equation}\label{eq:def Tkld}
T_{k,l,d}(p) := \int_{\mathcal{G}_{k,d}} \int_{\mathcal{G}_{l,d}} \langle P_V, P_W \rangle^p d\sigma_k(V) d\sigma_l(W).
\end{equation}
Again, the $O(\R^d)$-invariance of $\sigma_k$ implies 
\begin{equation*}
T_{k,l,d}(p)= \int_{\mathcal{G}_{k,d}} \langle P_V, P_W \rangle^p d\sigma_k(V), \quad \text{for all $W\in\mathcal{G}_{l,d}$}. 
\end{equation*}
To shorten notation, let $T_{k,d}(p):=T_{k,k,d}(p)$.

\begin{theorem}\label{Th designs}
Let $\{V_j\}_{j=1}^n\subset \mathcal{G}_{k,d}$ and $\sum_{j=1}^n \omega_j=1$.  We then have 
\begin{equation}\label{minimum FFP}
\FFP(\{(V_j,\omega_j)\}_{j=1}^n,p)\geq  T_{k,d}(p).
\end{equation}
Moreover, the following properties are equivalent:
\begin{enumerate}
\item[(1)] $\{(V_j,\omega_j)\}_{j=1}^n$  is  a cubature of strength $2p$  in $\Gkd$.
\item[(2)] For all $\mu$, $1\leq \deg(\mu)\leq p$, for all $f\in H_{k,d}^{2\mu}$, $\sum_{j=1}^n \omega_j f(V_j)=0$.
\item[(3)] For all $\mu$, $1\leq \deg(\mu)\leq p$, $\sum_{i,j=1}^n \omega_i\omega_j P_{\mu}(y_1(V_i,V_j),\dots, y_k(V_i,V_j))=0$.
\item[(4)] $\FFP(\{(V_j,\omega_j)\}_{j=1}^n,p)= T_{k,d}(p)$.
\item[(5)] There is a constant $A>0$ such that $\sum_{j=1}^n \omega_j \langle P_W,P_{V_j}\rangle^p = A$, for all $W\in\mathcal{G}_{k,d}$.
\item[(6)] There are constants $A_l>0$, $l=1,\ldots,k$, such that $\sum_{j=1}^n \omega_j\langle P_{W_l},P_{V_j}\rangle^p = A_l$, for all $W_l\in\mathcal{G}_{l,d}$.
\end{enumerate}
\end{theorem}

\begin{proof} 
The inequality \eqref{minimum FFP} follows from the positive definiteness of  the functions $s^p(V,W):=\langle P_W,P_V\rangle^p$. Indeed, $s$ is obviously positive definite, and the product of positive definite functions is again positive definite.  Moreover, every $O(\R^d)$-invariant positive definite function $F$ on $\Gkd$ is a non negative linear combination of the zonal polynomials $P_{\mu}$ in the variables $y_1(\cdot,\cdot),\ldots,y_k(\cdot,\cdot)$, i.e.,
\begin{equation*}
F(V,W) = \sum_\mu \lambda_\mu P_\mu(y_1(V,W),\ldots,y_k(V,W)),
\end{equation*}
where $\lambda_\mu\geq 0$ for all $\mu$. This important result goes back to \cite{Bochner:1941aa}. Since $(V,W)\mapsto P_\mu(y_1(V,W),\ldots,y_k(V,W))$ is positive definite, $F-\lambda_0$ is positive definite too, so 
\begin{equation*}
\sum_{i,j=1}^n \omega_i\omega_jF(V_i,V_j)\geq \lambda_0 \big(\sum_{j=1}^n\omega_j\big)^2,
\end{equation*}
 for all $ \{V_j\}_{j=1}^n\subset \Gkd$. For
$F= s^p$, we have $\lambda_0=T_{k,d}(p)$.

The equivalences between (1)-(4) have already been proven in \cite{Bachoc:2005aa,Bachoc:2002aa} for constant weights. Incorporating weights is straightforward so we omit it here.

(1)$\Rightarrow$(5): The mapping $V\mapsto \langle P_W,P_V\rangle^{p}$ is an element in $\Pol_{\leq 2p} (\Gkd)$, for all $W\in\mathcal{G}_{l,d}$ and $l=1,\ldots,k$. For $W\in\mathcal{G}_{k,d}$, the  property \eqref{def cubature formula} implies
\begin{equation*}
\sum_{j=1}^n \omega_j\langle P_W,P_{V_j}\rangle^p = \int_{\Gkd} \langle P_W,P_V\rangle^{p}d\sigma_k(V)=T_{k,d}(p)=:A.
\end{equation*} 
The implication (1)$\Rightarrow$(6) follows in the same way using $A_l=T_{l,k,d}(p)$. Since (6)$\Rightarrow$(5) is obvious, we only need to verify (5)$\Rightarrow$(1): As for \eqref{eq:specification of A}, we can compute $A=T_{k,d}(p)$. Therefore, we derive $\sum_{i,j}\omega_i\omega_j\langle P_{V_i},P_{V_j}\rangle^p = T_{k,d}(p)$, which implies (1).

\end{proof}

\begin{remark} A few comments are in order.
\begin{enumerate}
\item Since $\Pol^1_{\leq 2}(\mathcal{G}_{k,d})\subset \Pol_{\leq 2}(\mathcal{G}_{k,d})$, every cubature of strength $2p$ is a tight $p$-fusion frame according to Theorem \ref{Th weak designs}. 
In particular, the designs of strength $2p$ in Grassmann spaces provide an interesting subclass of tight $p$-fusion frames. 
\item We have already seen that $T_{k,d}(1)=k^2/d$. In \cite[Remark 6.4]{Bachoc:2005aa}, an explicit expression of $T_{k,d}(p)$ is given for $p=2,3$.  In general, $T_{k,d}(p)$ can be calculated from the expression of 
 $(y_1+\dots+y_k)^p$ as a linear combination of the zonal polynomials $P_{\mu}$, cf.~\cite[Lemma 6.2]{Bachoc:2005aa}.
 \item For $p=1$, Theorems \ref{Th designs} and \ref{prop:generalized simplex bound} show that the tight fusion frames  of equal dimension $k$ are exactly the cubatures of strength $2$  of $\Gkd$.
 \end{enumerate}
 \end{remark}

It is natural to ask for the existence of the objects discussed in this section, namely tight $p$-fusion frames and cubatures, and beyond existence, it is also desirable to discuss the size $n$ of these objects as a function of $p$ and $d$.
In these directions, the following results are borrowed  from  \cite{Harpe:2005aa}:

\begin{proposition}[\cite{Harpe:2005aa}]\label{prop:about cubature n}

\begin{enumerate}
\item There exists a tight $p$-fusion frame $\{(V_j,\omega_j)\}_{j=1}^n$ with $n\leq \dim(\Pol_{\leq 2p} ^1(\Gkd))-1=\binom{2p+d-1}{d-1}-1$.
\item There exists a cubature $\{(V_j,\omega_j)\}_{j=1}^n$ of strength $2p$ such that the inequality $n\leq \dim(\Pol_{\leq 2p}(\mathcal{G}_{k,d}))-1$ holds.
\item If $\{(V_j,\omega_j)\}_{j=1}^n$ is a cubature of strength $4p$, then $n\geq \dim(\Pol_{\leq 2p}(\mathcal{G}_{k,d}))$. 
\end{enumerate}
\end{proposition}

\begin{proof} 1. and 2.  follow from Proposition 2.6 and 2.7 in \cite{Harpe:2005aa}, and the fact that $\dim(V^{(2\ell)})=\binom{d+2\ell-1}{d-1}-\binom{d+2\ell-3}{d-1}$.
3. follows from Proposition 1.7 in \cite{Harpe:2005aa}.
\end{proof}

It should be noted that the existence statements above are non constructive by nature. In the next section,
some  explicit constructions are discussed. 

\begin{remark}
We aim to minimize the $p$-fusion frame potential among all collections of $k$-dimensional linear subspaces whose weights add up to one. Proposition \ref{prop:about cubature n} and Theorem \ref{Th designs} ensure that there exists a minimizer of cardinality less than $\dim(\Pol_{\leq 2p}(\mathcal{G}_{k,d}))$.
\end{remark}

%

\section{Some constructions of tight $p$-fusion frames}\label{sec:construction}
 In this section, we present three constructions of tight $p$-fusion frames. The first one is standard,  it uses orbits of finite subgroups of $O(\R^d)$ to construct 
 tight $p$-fusion frames of equal weights. This idea  has been extensively used for  the construction of codes and designs in many spaces (see e.g. \cite{SPLAG}, \cite{Ericson})  and also specifically in Grassmann spaces (\cite{Bachoc:2004aa},\cite{Conway:1999aa}, \cite{Creignou:2008aa}). It leads to many nice and explicit examples of highly symmetric tight $p$-fusion frames with equal weights and dimension, although only for small values of $p$. The second one relates tight $p$-fusion frames to designs in projective spaces. We show that the $p$-designs in complex and quaternionic spaces in the sense of \cite{Hoggar:1982aa} give rise to tight $p$-fusion frames of equal dimension  $2$ and $4$, respectively.  
The  last  construction  fits together tight $p$-fusion frames of different dimensions in a very simple way. It can be used, for example, to extend a tight $p$-frame for a lower dimensional space to a 
tight $p$-frame in a larger space, guided by the structure of a tight $p$-fusion frame for the larger space. These constructions are interrelated; $p$-designs in complex and quaternionic spaces 
can be constructed from orbits of complex, respectively quaternionic groups such as the reflection groups; in turn, they can be used as the building blocks in the last construction, together with tight $p$-frames in $\R^2$ or $\R^4$  in order to construct tight $p$-frames in larger dimensions.


\subsection{Tight $p$-fusion frames from orbits of finite subgroups of $O(\R^d)$.}  We address the following question: given a finite subgroup $G$ of $O(\R^d)$,
what property of $G$ would ensure that every orbit $G.V:=\{g(V) : g\in G\}$ on every Grassmann space $\Gkd$ is a tight $p$-fusion frame?  

We remark first that, if an orbit $G.V:=\{g(V) : g\in G\}$
is a tight $p$-fusion frame, then it satisfies the property \eqref{eq:p fusion def} with equal weights. To see this, one has to sum up the conditions \eqref{eq:p fusion def} for $x=g(y)$, when $g$ runs in $G$. 

The space $\R[\underline{X}]_{2p}$, $\underline{X}=(X_1,\dots,X_d)$,  of homogeneous polynomials in $d$ variables of degree $2p$, is endowed with the standard linear action of $O(\R^d)$. Let $\big(\R[\underline{X}]_{2p}\big)^G$ denote the collection of $P\in  \R[\underline{X}]_{2p}$ such that $P(\underline{X}M)=P(\underline{X})$,  for all $M\in G$. 

\begin{theorem} \label{th:irreducible group}
If $G$ is a finite subgroup of $O(\R^d)$, then the following conditions are equivalent:
\begin{enumerate}
\item[(1)] For all $1\leq k<d$ and all $V\in \Gkd$, the collection $G.V:=\{g(V) : g\in G\}$ is a tight $p$-fusion frame.
\item[(2)] $\big( \R[\underline{X}]_{2p}\big)^G=\R(X_1^2+\dots+X_d^2)^p$.
\end{enumerate}
\end{theorem}
\begin{remark}\label{remark:set and sequences}
Obviously $(X_1^2+\dots +X_d^2)^p$ is invariant by the orthogonal group so the condition (2) means that $G$ does not afford other invariant
polynomials than the ones which are invariant by the full orthogonal group.
\end{remark}

\begin{proof}[Proof of Theorem \ref{th:irreducible group}] It is a straightforward adaptation of the proof of \cite[Theorem 4.1]{Bachoc:2005aa}.  In view of (2) in Proposition \ref{prop equal dim} we can assume $k\leq d/2$.
We recall that 
\begin{equation*}
\Pol_{\leq 2p}^1(\Gkd)=\bigoplus_{\ell=0}^{p} H_{k,d}^{(2\ell)}\simeq \bigoplus_{\ell=0}^p V_d^{(2\ell)} \simeq \R[\underline{X}]_{2p}.
\end{equation*}
So the condition (2) is also equivalent to: $\big(H_{k,d}^{(2\ell)}\big)^G=\{0\}$ for all $k\leq d/2$ and for all $1\leq \ell\leq p$.

Let $V\in \Gkd$ and let $G_V$ denote the stabilizer of $V$ in $G$. For $f\in H_{k,d}^{(2\ell)}$, 
\begin{equation*}
\sum_{U\in G.V} f(U)=\frac{1}{|G_V|} \sum_{g\in G} f(g(V))=\Big(\sum_{g\in G} g.f \Big) (V),
\end{equation*}
where $g.f(V):=f(g(V))$. Since $\sum_{g\in G} g.f$ runs in $\big(H_{k,d}^{(2\ell)}\big)^G$, condition (2) is also equivalent to $\sum_{U\in G.V} f(U)=0$ for all $k$, $V\in \Gkd$ 
and $1\leq \ell\leq p$. From condition (2) in Theorem \ref{Th weak designs}   it amounts to the property that $G.V$ is a tight $p$-fusion frame for all $k$ and $V\in \Gkd$.
\end{proof}

\begin{example}
\begin{enumerate}
\item For $p=1$, condition (2) is equivalent to the irreducibility of $\R^d$ under the action of $G$, i.e., any orbit $Gx$ spans $\R^d$, for $0\neq x\in\R^d$. So
we recover the criterion of \cite{Vale:2005aa} for tight frames.  In \cite{Creignou:2008aa}, pairs $(G,H)$ such that $G$ acts irreducibly on $\R^d$ and two-transitively on $G/H$, are used to construct 
Grassmannian packings that are equiangular and meet the simplex bound (see the Section \ref{subsection:finite p frames} for these notions). Thus they  also provide tight frames.
\item For $p\geq 2$, standard examples are given by the Weyl groups of the root systems $A_2$, $D_4$, $E_6$, $E_7$ ($p=2$), $E_8$ ($p=3$), $H_4$ ($p=5$). An infinite family 
is provided by the real Clifford groups ${\mathcal C}_k\subset O(\R^{2^k})$ that satisfy (2) for $p=3$; some orbits of these groups on Grassmann spaces lead to good Grassmann codes  as described in \cite{Calderbank:1999uq}.
Another well-known example is the automorphism group of the Leech lattice $2.Co_1$,
a subgroup of $O(\R^{24})$ that holds the desired property for $p=5$. 
\item It should be noted that, if $-\Id\in G$, condition (2) is exactly the condition required for every $G$-orbits on the unit sphere $S^{d-1}$ to be \emph{$2p$-spherical designs} (being antipodal, these designs are trivially of strength $(2p+1)$).
These groups are  called \emph{$2p$-homogeneous} in \cite{Harpe:2004aa}. A useful sufficient condition for a group $G$ to be $2p$-homogeneous, due to E. Bannai, is that the restrictions 
to $G$ of the $O(\R^d)$-representations $V_d^{(k)}$ for $1\leq k\leq p$ are irreducible. We refer to \cite{Harpe:2004aa} for a proof of this result and for more properties of homogeneous groups. See also \cite{Lempken:2001aa} for a classification of the quasi-simple groups
such that $V_d^{(2)}$ is irreducible.
\item In \cite{Tiep:2006aa} a complete classification of the finite groups $G\subset O(\R^d)$
such that $(V_d^{(2)})^G=(V_d^{(4)})^G=(V_d^{(2,2)})^G=\{0\}$, is given. The assumption is here slightly stronger than  (2) for $p=2$; it arises naturally 
in the study of  designs in Grassmann spaces \cite{Bachoc:2002aa, Bachoc:2005aa}.
\end{enumerate}
\end{example}

\subsection{Tight $p$-fusion frames from $p$-designs in projective spaces}

The notion of $p$-designs has been developed in a uniform setting for the connected, compact, symmetric spaces of rank one \cite{Delsarte:1975fk, Neumaier:1981aa, Hoggar:1982aa}. These spaces include the projective spaces over the real, complex and quaternionic fields (the unit spheres of Euclidean space, and the projective plane over the octonions, in fact, make the list of rank one, connected, compact, symmetric spaces complete). 
For  $K=\R,\C,{\mathbb H}$, a subset  $\{P_j\}_{j=1}^n\subset {\mathcal P}(K^d)$ of the projective space is a $p$-design if
\begin{equation*}
\int_{{\mathcal P}(K^d)} f(V) d\sigma(V)=\frac{1}{n}\sum_{j=1}^n f(P_j)
\end{equation*}
for all functions $f\in \Pol_{\leq p}({\mathcal P}(K^d))$, where $\sigma$ denotes the normalized Lebesgue measure on ${\mathcal P}(K^d)$ and $\Pol_{\leq p}({\mathcal P}(K^d))$ is a 
subspace of functions on ${\mathcal P}(K^d)$, which are polynomial of degree bounded by $p$ in some reasonable sense. 

It should be noted that, unfortunately, the notations disagree between \cite{Hoggar:1982aa} and \cite{Bachoc:2002aa} in the case ${\mathcal G}_{1,d}={\mathcal P}(\R^d)$,
so that what is called a $p$-design in the projective setting \cite{Hoggar:1982aa} corresponds to a $2p$-design in \cite{Bachoc:2002aa}. 

Let $h(x,y)=\sum_{i=1}^d x_i \overline{y_i}$ denote the standard hermitian form on $K^d$ (where the conjugation on $\R$ is the identity). For $P_1$, $P_2$ in ${\mathcal P}(K^d)$, 
we define $t(P_1,P_2)$  to be the common value of $|h(x_1,x_2)|^2$ for any  $x_i\in P_i$, $h(x_i,x_i)=1$. If $\Pol_{\leq p}({\mathcal P}(K^d))$ is the span of $\{ P\mapsto t(P,Q)^p : Q\in  \mathcal{P}(K^d) \}$, then equivalently, 
$\{P_j\}_{j=1}^n\subset {\mathcal P}(K^d)$ is a $p$-design if, for some constant $A_p$, 
\begin{equation}\label{projective design}
\sum_{j=1}^n t(P_j, P)^p= A_p\quad \text{ for all }P\in {\mathcal P}(K^d).
\end{equation}

Now, we make the usual identification of $\C^d$ with $\R^{2d}$ and ${\mathbb H}^d$ with $\R^{4d}$, noticing that $h(x,x)=\Vert x \Vert^2$.
Obviously, for $x\in K^d$, and $P_j=Kx_j$, $h(x,x)=h(x_j,x_j)=1$, we have $\Vert P_{P_j}(x)\Vert^2=t(P_j,Kx)$ so that \eqref{projective design}
amounts to the property of tight $p$-fusion frames  $\sum_{j=1}^n \Vert P_{P_j}(x)\Vert^{2p}= A_p$ for $x\in K^d$, $\Vert x\Vert^2=1$.  We have proved:

\begin{theorem}
A $p$-design in the projective space ${\mathcal P}(\C^d)$ (respectively ${\mathcal P}({\mathbb H}^d)$) is a tight $p$-fusion frame in $\R^{2d}$ with subspaces of equal dimension $2$
(respectively in $\R^{4d}$ with subspaces of equal dimension $4$). 
\end{theorem}

Many examples of projective $p$-designs for $p\leq 5$ are described in \cite{Hoggar:1982aa}. 
Most of them are related to complex or quaternionic reflection groups.

\subsection{Extension and refinement of tight $p$-fusion frames}
We consider the following construction. Let  $\FFF_0=\{(V_j,v_j)\}_{j=1}^n$ and $\FFF_1=\{(W_i,\omega_i)\}_{i=1}^m$ be two sets of weighted linear subspaces, such that
$V_j\subset \R^{\ell}$, $\ell< d$ and $W_i\in \Gld$. 
Let $f_i: \R^{\ell}\to W_i$ be some fixed isometries, and let $V_{i,j}:=f_i(V_j)\subset W_i$.  Then, for all $i$, $\{V_{i,j}\}_{j=1}^n$ is a collection of linear subspaces of $W_i$ which is isometric to $\FFF_0$.
We now consider:  
\begin{equation}\label{def: fit}
\FFF:= \{(V_{i,j}, \omega_i v_j)\}_{\substack{1\leq i\leq  m\\ 1\leq j\leq n}}.
\end{equation}

\begin{theorem}\label{th:fitting together} 
If $\FFF_0$ and $\FFF_1$ are   tight $p$-fusion frames, then $\FFF$ is also   a tight $p$-fusion frame.
\end{theorem}

\begin{proof} Let $x\in \R^d$, we want to compute $\sum_{i=1}^m\sum_{j=1}^n \omega_i v_j\| P_{V_{i,j}}(x)\|^{2p}$. Because $\FFF_0$ is assumed to be a tight $p$-fusion frame, there exists $A_{\FFF_0}$ such that, 
for all $1\leq i\leq m$, for all $y\in W_i$, 
\begin{equation*}
\sum_{j=1}^n  v_j\| P_{V_{i,j}}(y)\|^{2p} =A_{\FFF_0} \| y \| ^{2p}.
\end{equation*}
Also $\FFF_1$ is a tight $p$-fusion frame so, for some $A_{\FFF_1}$, 
\begin{equation*}
\sum_{i=1}^m \omega_i \| P_{W_i}(x)\| ^{2p} =A_{\FFF_1} \|x\|^{2p}.
\end{equation*}
Let $x_i:=P_{W_i}(x)$. Then, $P_{V_{i,j}}(x)=P_{V_{i,j}}(x_i)$. So,
\begin{align*}
\sum_{i=1}^m\sum_{j=1}^n \omega_i v_j\| P_{V_{i,j}}(x)\|^{2p}  &=\sum_{i=1}^m \omega_i\Big(\sum_{j=1}^n  v_j\| P_{V_{i,j}}(x_i)\|^{2p} \Big)\\
&= \sum_{i=1}^m \omega_i A_{\FFF_0} \| x_i \| ^{2p}\\
&= A_{\FFF_0}\sum_{i=1}^n \omega_i\| P_{W_i}(x)\| ^{2p} = A_{\FFF_0}A_{\FFF_1} \|x\|^{2p}.\qedhere
\end{align*}

\end{proof}

\noindent{\bf Example:} One can take for $\FFF_0$ a tight $p$-frame in dimension $\ell$; the resulting collection $\FFF$ is a tight $p$-frame in dimension $d$ with $nm$ elements. Depending on the perspective, the latter extends a tight $p$-frame to a larger dimensional space or it refines a tight $p$-fusion frame by subdividing its subspaces into smaller subspaces.

\section{An unrestricted lower bound for the $p$-fusion frame potential}\label{sec:p pot}
In this section, we generalize the inequality \eqref{minimum FFP} for the $p$-fusion frame potential in Definition \ref{def:ffp} when the subspaces $V_j$ are not restricted to have the same dimension. To that end, we will exploit the $O(\R^d)$-decomposition of
the Hilbert space $L^2(\GG_d)$ and the structure of positive definite functions of this space. In contrast with the case of equal dimensions,  difficulties arise from the fact that the irreducible representations of $O(\R^d)$ occur in $L^2(\GG_d)$ with non trivial multiplicities.
In particular, $L:= \sum_{k=1}^{d-1} \C {\bf 1}_{\mathcal{G}_{k,d}}$,  the subspace of functions  taking constant values on $\mathcal{G}_{k,d}$ is isomorphic to $(d-1)$ copies of the trivial representation. 
So, we have to replace the single coefficient $\lambda_0=T_{k,d}(p)$ that occurred in \eqref{minimum FFP} with a matrix of size $(d-1)\times (d-1)$, i.e., let $T_d(p) := \big(T_{k,l,d}(p) \big)_{k,l}$, where $T_{k,l,d}(p)$ is as in \eqref{eq:def Tkld}. 
\begin{theorem}\label{theorem:main pot estimate}
Given positive weights $\{\omega_j\}_{j=1}^n$, let $\{V_j\}_{j=1}^n\subset \mathcal{G}_d$ and $p\geq 1$ be an integer. Define $m_k:=\sum_{\dim(V_i)=k} \omega_i$. 
If $M=(m_1,\ldots,m_{d-1})$, then 
\begin{equation}\label{eq:main finite estimate}
\FFP(\{(V_j,\omega_j)\}_{j=1}^n,p)\geq M T_d(p) M^\top.
\end{equation}
\end{theorem}
We recall that a function $F\in L^2(\mathcal{G}_d\times \mathcal{G}_d)$ is said to be \emph{positive definite} if, for all $f\in L^2(\mathcal{G}_d)$, 
\begin{equation*}
\int_{\mathcal{G}_{d}} \int_{\mathcal{G}_{d}} F(V,W) f(V) \overline{f(W)}d\sigma(V) d\sigma(W)\geq 0. 
\end{equation*}
For a continuous function $F:\mathcal{G}_d\times \mathcal{G}_d\rightarrow \C$, it amounts to ask that, for all integers $s\geq 1$, and $(W_1,\ldots, W_s)\in (\mathcal{G}_d)^s$, the matrix $\big( F(W_i,W_j) \big)_{i,j}$ is hermitian positive semi-definite. In order to prove Theorem \ref{theorem:main pot estimate}, we need some preparation with the following lemma. We  use the tensor notation $L\otimes L = \{(V,W)\mapsto f(V)g(W) : f,g\in L\}$. 
\begin{lemma}\label{lemma:for theorem}
Let $s^p(V,W):=\langle P_V,P_W \rangle^p$. Then, $s^p=F_0+F_1$, with $F_0\in L\otimes L$, $F_1\in L^{\perp}\otimes L^\perp$, and $F_0$, $F_1$ are positive definite functions.
\end{lemma}
\begin{proof}
The function $s^p(V,W)$ is contained in $L^2(\mathcal{G}_d)\otimes L^2(\mathcal{G}_d)= (L\oplus L^\perp) \otimes (L\oplus L^\perp)$. Let $f_0\in L$ and $f_1\in L^{\perp}$. We observe that $F$ and $f_0$ are $O(\R^d)$-invariant. We have
\begin{align*}
\int_{\GG_d}\int_{\GG_d}
s^p(V,W)f_0(V)\overline{f_1(W)}&d\sigma(V)d\sigma(W)\\
& \hspace{-1.5cm}=\int_{O(\R^d)}\int_{\GG_d}\int_{\GG_d}
s^p(V,W)f_0(V)\overline{f_1(gW)}d\sigma(V)d\sigma(W) dg\\
&\hspace{-1.5cm} =\int_{\GG_d}\int_{\GG_d} s^p(V,W)f_0(V)\Big(\int_{O(\R^d)}\overline{f_1(gW)}dg\Big)d\sigma(V)d\sigma(W).
\end{align*}
But, for $W\in \GG_{k,d}$, since $\GG_{k,d}$ is $O(\R^d)$-homogeneous,
\begin{equation*}
\int_{O(\R^d)}\overline{f_1(gW)}dg =\gamma_k \int_{\GG_{k,d}} \overline{f_1(W)}d\sigma_k(W)
\end{equation*}
for some constant $\gamma_k$. Because $f_1\in L^{\perp}$, the right
hand side is zero. Therefore, we have $s^p=F_0+F_1\in (L\otimes L) \oplus (L^\perp\otimes L^\perp)$.

The functions $s^p(V,W)$ are positive definite on $\GG_d$. So, for all $f=f_0+f_1\in L\oplus L^{\perp}$, we obtain
\begin{align*}
\int_{\GG_d}\int_{\GG_d}
F_0(V,W)f(V)\overline{f(W)}&d\sigma(V)d\sigma(W)\\
&=\int_{\GG_d}\int_{\GG_d}F_0(V,W)f_0(V)\overline{f_0(W)}d\sigma(V)d\sigma(W)\\
&=\int_{\GG_d}\int_{\GG_d}s^p(V,W)f_0(V)\overline{f_0(W)}d\sigma(V)d\sigma(W)\geq 0.
\end{align*}
Hence, $F_0$ is positive definite. A similar argument shows that $F_1$ is also positive definite.
\end{proof}

\begin{proof}[Proof of Theorem \ref{theorem:main pot estimate}]
Since the set of functions $\{\1_{\GG_{k,d}}(V)\1_{\GG_{l,d}}(W), 1\leq k,l\leq d-1\}$ 
is an orthonormal basis of $L\otimes L$, we have 
\begin{equation*}
F_0(V,W)=\sum_{k,l} T_{k,l,d}(p) \1_{\GG_{k,d}}(V)\1_{\GG_{l,d}}(W).
\end{equation*}
Since $s^p-F_0$ is positive definite, we obtain 
$
\sum_{i,j}\omega_i(s^p(V_i,V_j)-F_0(V_i,V_j))\omega_j \geq 0 
$, 
and this yields
\begin{align*}
\sum_{i,j}\omega_i\omega_j\langle P_{V_i}, P_{V_j} \rangle^p & \geq \sum_{i,j} \sum_{k,l} \omega_i\omega_jT_{k,l,d}(p) {\bf 1}_{\mathcal{G}_{k,d}}(V_i){\bf 1}_{\mathcal{G}_{l,d}}(V_j)\\
& =  \sum_{k,l} T_{k,l,d}(p)\sum_{i,j} \omega_i\omega_j{\bf 1}_{\mathcal{G}_{k,d}}(V_i){\bf 1}_{\mathcal{G}_{l,d}}(V_j)\\
& =  \sum_{k,l}  m_k m_lT_{k,l,d}(p) = M T_d(p) M^\top.
\end{align*}
\end{proof}

\begin{remark}\label{label:remark}
\begin{enumerate}

\item The measures $d\sigma_l d\sigma_k$ induce  measures $d\lambda_{l,k}$ on $[0,1]^l$ for the variables  $y_i=\cos^2(\theta_i(V,W))$,  which are computed in \cite{James:1974aa}. Up to a multiplicative constant, one has, for $l\leq k\leq d/2$, 
\begin{equation*}
d\lambda_{l,k}=\int_{[0,1]^l} \prod_{1\leq i<j\leq l} |y_i-y_j|\prod_{i=1}^l
  y_i^{(k-l-1)/2}(1-y_i)^{(d-k-l-1)/2}dy_i.
\end{equation*}
Note that $\lambda_{1,k}$ has already occurred in the proof of Theorem \ref{Th weak designs}.
The
\emph{zonal spherical intertwining polynomials}  for the Grassmann spaces $\GG_{l,d}$ and $\GG_{k,d}$, denoted $P_{\mu}^{l,k}(y_1,\dots,y_k)$,
 are symmetric   polynomials, and are orthogonal  for the measure $\lambda_{l,k}$ (\cite{James:1974aa}).  They are indexed by the partitions $\mu$ of length at most $l$. These polynomials already occurred in Section \ref{sec:equal dim} for $(l,k)=(k,k)$ and for $(l,k)=(1,k)$. 

Since $\langle P_V,P_W \rangle^p=(y_1+\dots +y_k)^p$, the number
  $T_{l,k,d}(p)$ 
corresponds to the constant term in the expression of
$(y_1+\dots+y_k)^p$ as a linear combination of  these polynomials. For example,  $P_{(1)}^{l,k}=(y_1+\dots
+y_l)-lk/d$ (up to a multiplicative factor) and thus
$T_{l,k,d}(1)=lk/d$. Knowledge of these polynomials allows to give an explicit expression for
$T_{l,k,d}(p)$. We observe  that, because these polynomials have rational
coefficients, $T_{l,k,d}(p)$
is a rational function of $l,k,p$. 
\item For $p=1$, we have
\begin{align*}
\sum_{1\leq l,k\leq d-1}T_{l,k,d}(1) m_lm_k&= 
\sum_{1\leq l,k\leq d-1}(lk/d) m_lm_k\\
&=\frac{1}{d} \big(\sum_{1\leq k\leq d-1}  k m_k\big)^2\\
&=\frac{1}{d} \big(\sum_{j=1}^n \omega_j\dim(V_j)\big)^2.
\end{align*}
Thus, we recover the lower bound for $p=1$ in Theorem \ref{prop:generalized simplex bound}.
\item For $p\geq 1$, if $\{V_j\}_{j=1}^n\subset \Gkd$, then $M=(0,\dots, m_k=\sum_{j=1}^n\omega_j,0,\dots)$ and $MT(p)M^\top=\big(\sum_{j=1}^n\omega_j\big)^2T_{k,d}(p)$ so that we recover \eqref{minimum FFP}.

\end{enumerate}
\end{remark}

\section*{Acknowledgements} 
The authors would like to thank the referees for their suggestions that improved the presentation of this paper. 
M.~E.~is supported by the NIH/DFG Research Career Transition Awards Program (EH 405/1-1/575910).

%
 
\end{document}